\documentclass[11pt]{amsart}

\usepackage{epsfig}
\usepackage{amsmath}
\usepackage{amssymb}
\usepackage{graphicx}

\usepackage{hyperref}
\usepackage{a4wide}
 \usepackage{amssymb,amscd, amsmath}
\usepackage{pst-node}
\usepackage[all,cmtip]{xy}

\theoremstyle{plain}
\newtheorem{prp}{Proposition}[section]
\newtheorem{lemma}[prp]{Lemma}
\newtheorem{stat}[prp]{Claim}

\newtheorem{theorem}[prp]{Theorem}
\theoremstyle{definition}

\newtheorem{remark}[prp]{Remark}

\newtheorem{cstr}[prp]{Construction}

\usepackage{indentfirst,latexsym,bm}

 \newcommand\soc{{\sf Soc}}
\newcommand\aut{{\sf Aut}} 
 
\newcommand\sym{{\sf Sym}}

\DeclareMathOperator{\Cos}{Cos}

\renewcommand{\leq}{\leqslant}

\newcommand{\F}{\mathrm{F}}
\newcommand{\E}{\mathrm{E}}

\usepackage[left=2cm,right=2cm,
    top=2cm,bottom=2.5cm,bindingoffset=0cm]{geometry}

\begin{document}

 \title[A new infinite family of covers of complete graphs that admit a half-transitive action...]
{A new infinite family of covers of complete graphs that admit a half-transitive action of a Suzuki simple group}

 \author[L.~Yu.~Tsiovkina]{L.Yu. Tsiovkina}
\address{N.N.~Krasovskii Institute of Mathematics and Mechanics of UB RAS, 16 S.~Kovalevskaya str., Yekaterinburg, 620990 Russia,\\
Ural Mathematical Center, Ural Federal University, 51 Lenina Str., Yekaterinburg, 620002 Russia}
 \email{l.tsiovkina@gmail.com}

\maketitle

\begin{abstract}
We construct a new infinite family of antipodal $2(q-1)$-covers of complete graphs
$K_{q^2+1}$ of diameter $3$, where $q=2^{2e+1}\ge 8$,  admitting a half-transitive action of  the Suzuki simple group  $Sz(q)$. We further show that each constructed cover  admits an involutory automorphism interchanging the two $Sz(q)$-orbits on arcs, which yields that its full automorphism group  is arc-transitive.
\end{abstract}

\vspace{2mm}

\medskip


\section{Introduction}\label{Introduction}

Antipodal covers of complete graphs form an important class of graphs, which has links to coding theory (e.g., see \cite{KrotovPotapov}),  distance-regular graphs  and  systems of equiangular lines and tight frames (e.g. see \cite{CoutinhoGodsilShiraziZhan}). These connections, together with the rich symmetry properties of antipodal covers, provide further motivation for their construction and classification.
Considerable attention has been devoted to antipodal covers of complete graphs of diameter $3$ satisfying additional symmetry conditions, most notably edge-transitivity and distance-regularity; see \cite{Tsi22c} for an overview. In the present paper, we drop the distance-regularity assumption and study edge-transitive antipodal covers of diameter $3$ in a more general setting.

It is easy to see that every edge-transitive group $G$ of automorphisms of an antipodal diameter-$3$ cover induces a $2$-homo\-ge\-neous permutation group $G^{\Sigma}$ on the set $\Sigma$ of antipodal fibres. Such a permutation group is known to be  either of affine type or almost simple due to a combination of  Burnside's and Kantor's theorems.  In the case where $G^{\Sigma}$ is almost simple, all distance-regular covers were classified, and several new constructions of non-distance-regular covers admitting a simple arc-transitive group of automorphisms were obtained in \cite{Tsi22c}. In contrast, to the best of our knowledge, no infinite family of antipodal diameter-$3$ covers admitting a simple half-transitive group of automorphisms has previously been constructed. Recall that a group of automorphisms of a graph is called \emph{half-transitive} if it is transitive on both vertices and edges but not on arcs.

The purpose of the present paper is to construct such a family. More precisely, for every Suzuki simple group $Sz(q)$, where $q=2^{2e+1}\ge 8$, we construct an antipodal cover of $K_{q^2+1}$ of diameter $3$, with fibres of size $2(q-1)$, on which $Sz(q)$ acts half-transitively.

The construction may be viewed as an index-$2$ refinement of a known distance-regular antipodal cover associated with $Sz(q)$. Starting with a Sylow $2$-subgroup $S$ of $Sz(q)$ and passing to a subgroup $S_1$ of index $2$ in $S$, each right $S$-coset splits into two right $S_1$-cosets. The resulting coset graph retains the vertex- and edge-transitivity of the Suzuki group, while its arc-transitivity is lost. Thus, the construction provides a natural mechanism for obtaining a half-transitive action of a non-abelian finite simple group on an antipodal cover of a complete graph.

Interestingly, although the Suzuki group itself acts half-transitively on the resulting graph, the graph admits a larger arc-transitive group of automorphisms. This distinction between the action of the simple group $Sz(q)$ and the symmetry of the graph as a whole is one of the characteristic features of the construction.

We now describe the construction.

\begin{cstr}\label{cstr}
Let $$G=Sz(q), \qquad q=2^{2e+1}\ge 8, $$ and let $S$ be a Sylow $2$-subgroup of $G$, so that $|S|=q^2$. Put $$ M:=N_G(S).$$

Let $S_1$ be a subgroup of $S$ of index $2$.

Choose an involution $g\in G\setminus M$ and an element $y\in S\setminus S_1$ for which there exists an involution $a\in S$ such that $(y,a)$ is the unique pair of elements with these properties satisfying
\begin{equation}\label{StrIde}
gag=y^{-1}gy.
\end{equation}
By the structural properties of the Suzuki groups described in \cite[Section~6]{Suz62}, one then has
$$y^2=a\qquad\text{and}\qquad |yg|=4.$$

In particular, $y$ has order $4$.

Set
\begin{itemize}
\item[(i)] $\Omega:=G/S_1$, the set of right $S_1$-cosets in $G$;
\item[(ii)] $K:=M\cap M^g$;
\item[(iii)] $$D:=S_1(yg)S_1\cup S_1(yg)^{-1}S_1.$$

\end{itemize}

We identify $G$ with the permutation group induced by its right multiplication action on $\Omega$.

Since $[S:S_1]=2$ and $y\notin S_1$, we have $$S=S_1\cup S_1y.$$

Consequently, $$B:=\{S_1,S_1y\}$$ is a block of imprimitivity for the action of $G$ on $\Omega$. Let $\mathcal B$ denote the corresponding system of imprimitivity containing $B$. The action induced by $G$ on $\mathcal B$ is permutation-isomorphic to the natural right multiplication action of $G$ on $$\Omega':=G/S, $$ the set of right $S$-cosets in $G$.

For $x\in G$, write $$B^x=\{S_1x,S_1yx\}.$$

In particular,  $$B^f=\{S_1f,S_1yf\}\qquad (f\in K),$$ and $$ B^g=\{S_1g,S_1yg\}.$$
Finally, let $$\Gamma:=\Cos(G,S_1,D)$$ be the graph with vertex set $\Omega$ in which
$$S_1x\sim S_1z \quad\Longleftrightarrow\quad xz^{-1}\in D.$$

\end{cstr}

Our main result is the following.

\begin{theorem}\label{mainthm}
For every choice of $S$, $S_1$, $g$, and $y$ satisfying the conditions of Construction~\ref{cstr}, the graph $\Gamma$ is a $G$-half-transitive antipodal $2(q-1)$-cover of $K_{q^2+1}$ of diameter $3$. Moreover, the full automorphism group $\operatorname{Aut}(\Gamma)$ is arc-transitive and is isomorphic to a subgroup of
$Z_2\times \operatorname{Aut}(Sz(q)).$

\end{theorem}

\section{Notation and preliminaries}

Throughout the paper, all graphs are finite, simple, and undirected. For a graph $\Gamma$, we denote its vertex set, edge set, and arc set by
$$V\Gamma,\qquad E\Gamma,\qquad A\Gamma,$$ respectively, and write $d_\Gamma(u,v)$ for the distance between vertices $u$ and $v$. For $i\ge 0$ and $v\in V\Gamma$, put
$$\Gamma_i(v):=\{u\in V\Gamma\mid d_\Gamma(u,v)=i\}.$$

If $X\le \operatorname{Aut}(\Gamma)$ is transitive on both vertices and edges but not on arcs, then $X$ is said to be \emph{half-transitive} on $\Gamma$.

A graph $\Gamma$ of diameter $d$ is called \emph{antipodal} if the relation
''to coincide or to be at distance $d$'' is an equivalence relation on $V\Gamma$. Its equivalence classes are called the \emph{antipodal classes} or \emph{fibres}.

An $r$-cover of $K_n$ is a graph whose vertex set is partitioned into $n$ fibres of size $r$ such that each fibre is a coclique and every two distinct fibres induce a perfect matching. Such a cover is antipodal of diameter $3$ if and only if every two non-adjacent vertices belonging to distinct fibres have a common neighbour; see \cite[Lemma~3.1]{GodsilHensel}.

For a group $X$, we write $$X^\#:=X\setminus\{1\}.$$

\begin{lemma}[\cite{Tsi15}]\label{Lemma1}
Let $\Gamma' := \Cos(G, S, g)$ be  the graph on the set $\Omega'$, in which
$$S x \sim S z \iff x z^{-1} \in S g S.$$ Then $\Gamma'$ is a distance-regular cover with parameters $(q^2 + 1, q - 1, q + 1)$, each fibre of which has the form
$$\{ S g^{\varepsilon} f s \in \Omega' \mid f \in K \},$$ where $s \in S$ and $\varepsilon \in \{0, 1\}$.
\end{lemma}

\begin{lemma}\label{st1htc}
$\Gamma$ is an undirected, connected, simple graph, and $G \leq \operatorname{Aut}(\Gamma)$ is its half-transitive group of automorphisms.
\end{lemma}

\begin{proof}
Since $D = D^{-1}$, $\Gamma$ is an undirected graph. Consider the element $y g$. Since the involution $y g y g = a g a$ is contained in neither of the maximal subgroups $M$, $M^g$, nor $N_G(K)$, by \cite{Suz62} we have $\langle S_1, y g \rangle = G$, hence $\Gamma$ is connected.

Furthermore, if an element $x$ inverts the edge $\{S_1, S_1 y g\}$, then its square belongs to the stabiliser of the two fibres $F(S_1)$ and $F(S_1 g)$ in $G$, i.e., to $K$. This implies $x \in N_G(K) = \langle K, g \rangle$, and $x = f g$ for some $f \in K$. But then $S_1 y g f g = S_1$, contradicting the initial assumption. Therefore, $G$ acts half-transitively.
\end{proof}

\begin{lemma}\label{Lemma2}
$\Gamma$ is a $2(q - 1)$-cover of the complete graph on $q^2 + 1$ vertices, with each fibre of the form
$$\{ S_1 y^{\delta} g^{\varepsilon} f s \in \Omega \mid f \in K, \delta \in \{0, 1\} \},$$ where $s \in S$ and $\varepsilon \in \{0, 1\}$.
\end{lemma}

\begin{proof}
It is clear that, due to the transitivity of $G$ on $\Omega$, $\Gamma$ is a $2$-cover of the graph $\Gamma'$, since otherwise, the adjacency $S_1 \sim S_1 y$ would imply $y \in D$, which contradicts the choice of $g\notin M$.

Note that the edge $\{S, S g\}$ of $\Gamma'$ corresponds exactly to two edges $\{S_1, S_1 y g\}$ and $\{S_1 y, S_1 g\}$ in the extension $\Gamma$.
Since $\Gamma'$ is a $(q^2 + 1, q - 1, q + 1)$-cover, and given the transitivity of $G$ on the edges of $\Gamma'$, we conclude that $\Gamma$ is a $2(q - 1)$-cover of the complete graph on $q^2 + 1$ vertices. The partition $\Sigma = \{\{ S_1 y^{\delta} g^{\varepsilon} f s \in \Omega \mid f \in K, \delta \in \{0, 1\}  \}\ |\ s \in S,  \varepsilon \in \{0, 1\}\}$ forms the set of fibres of $\Gamma$, because each element of $\Sigma$ induces a $2(q - 1)$-coclique in $\Gamma$, and the union of any two distinct elements of $\Sigma$ induces a perfect matching.
\end{proof}

Let $\sigma, \tau : S^{\#} \to S$ and $\eta : S^{\#} \to K$ be three functions defined by the rule
\begin{equation}\label{tausigmaeta}
g s g = \sigma(s) \eta(s) g \tau(s)
\end{equation}
for all $s \in S^{\#}$.

\begin{lemma}[see  {\cite[Proposition 5]{Suz62}}]\label{Lemma3}
The functions $\sigma$, $\tau$, and $\eta$ satisfy the following conditions:
\begin{enumerate}
\item $\tau(\tau(s)) = s$,
\item $\sigma(s) = \tau(s^{-1})^{-1}$,
\item $\tau(h^{-1} s h) = h \tau(s) h^{-1}$,
\item $\sigma(h^{-1} s h) = h \sigma(s) h^{-1}$,
\item $\eta(h^{-1} s h) = h^2 \eta(s)$,
\item $\tau(\sigma(s)) = \eta(s)^{-1} \tau(s)^{-1} \eta(s)$,
\item $\tau(c s) = \eta(s) \tau(\tau(c) \sigma(s)) \eta(s)^{-1} \tau(s)$,
\end{enumerate}
where $s, c \in S$, $s c \neq 1$, and $h \in K$.
\end{lemma}

\begin{lemma}\label{Lemma4} For all $s\in S-\{1,y^{-1}\}$  and $h\in K$, the following identities are true:
\begin{enumerate}

\item $\sigma(a^h) = (y^{-1})^{h^{-1}}, \quad \tau(a^h) = y^{h^{-1}}, \quad \eta(a^h) = h^2,$
\item $\sigma(y^h) = (y^{-1}a)^{h^{-1}}, \quad \tau(y^h) = a^{h^{-1}}, \quad \eta(y^h) = h^2,$
\item $\sigma((y^{-1})^h) = a^{h^{-1}}, \quad \tau((y^{-1})^h) = (ya)^{h^{-1}}, \quad \eta({(y^{-1})}^h) = h^2,$
\item $\sigma(ys) = y^{-1} a \, \sigma(a \sigma(s)), \quad \tau(ys) = \tau(a \sigma(s))^{\eta(s)^{-1}} \tau(s), \quad \eta(ys) = \eta(a \sigma(s)) \eta(s),$
\item $\sigma(y (y^{-1})^h) = y \, \sigma(aa^{h^{-1}}), \quad \tau(y (y^{-1})^h) = \tau(aa^{h^{-1}})^{h^{-2}} (y^{-1})^{h^{-1}}, \quad \eta(y (y^{-1})^h) = \eta(aa^{h^{-1}}) h^2.$
\end{enumerate}
\end{lemma}
\begin{proof}
 Obviously,
$$ga^h g = h g a g h^{-1} = h y^{-1} g y h^{-1} = (y^{-1})^{h^{-1}} h^2 g y^{h^{-1}}.$$
So $(1)$ follows from the structure identity \eqref{StrIde}.
 To prove $(2)$, it suffices to notice that
$$ygyg = y^2 y^{-1} g y g = a g a g = a g a.$$
Therefore,
$$g y^h g = h g y g h^{-1} = h y^{-1} a g a h^{-1} = (y^{-1} a)^{h^{-1}} h^2 g a^{h^{-1}}.$$
Similarly,
$$g (y^{-1})^h g = (h g y g h^{-1})^{-1} = a^{h^{-1}} h^2 g (y a)^{h^{-1}}.$$
This shows $(3)$.

From the above statements, it also follows that
$$\tau(aa^{h^{-1}}) = \tau(a^{h^{-1}} a) = \eta(a) \tau( \tau(a^{h^{-1}}) \sigma(a) ) \eta(a)^{-1} \tau(a) = \tau( y^h y^{-1} ),$$
and
$$\sigma(aa^{h^{-1}}) = \tau(aa^{h^{-1}})^{-1} = y^{-1} \tau( y^h y^{-1} )^{-1}.$$
 We also have
$$g y s g = y^{-1} a g a \sigma(s) g \eta(s)^{-1} \tau(s) =
y^{-1} a \sigma(a \sigma(s)) \eta(a \sigma(s)) \eta(s) g \tau(a \sigma(s))^{\eta(s)^{-1}} \tau(s).$$
In particular,
\begin{multline*}
g y (y^{-1})^h g =
y^{-1} a \sigma(a \sigma((y^{-1})^h)) \eta(a \sigma((y^{-1})^h)) \eta((y^{-1})^h) g \tau(a \sigma((y^{-1})^h))^{\eta((y^{-1})^h)^{-1}} \tau((y^{-1})^h) = \\
y \sigma(aa^{h^{-1}}) \eta(aa^{h^{-1}}) h^2 g \tau(aa^{h^{-1}})^{h^{-2}} (y^{-1})^{h^{-1}}.
\end{multline*}
This shows $(4)$ and $(5)$.
\end{proof}

\section{Proof of main results} \label{proof}
To determine the diameter of $\Gamma$ and some other its combinatorial properties, for each element $f \in K$, we will consider the following four sets:
$${\mathcal M}_{gf}^1 = \{ s \in {S_1}^\# \mid \eta(s) = f \text{ and } (\sigma(s) \in S_1, \ \tau(s)^f \in S_1 \text{ or } \sigma(s) \in S_1 y, \ \tau(s)^f \in S_1 y) \},$$
$${\mathcal M}_{gf}^2 = \{ r \in {S_1 y} \mid \eta(r) = f \text{ and } (\sigma(r) \in S_1 y, \ \tau(r)^f \in S_1 \text{ or } \sigma(r) \in S_1, \ \tau(r)^f \in S_1 y) \},$$
$${\mathcal M}_{ygf}^1 = \{ s \in {S_1}^\# \mid \eta(s) = f \text{ and } (\sigma(s) \in S_1 y, \ \tau(s)^f \in S_1 \text{ or } \sigma(s) \in S_1, \ \tau(s)^f \in S_1 y) \},$$
$${\mathcal M}_{ygf}^2 = \{ r \in {S_1 y} \mid \eta(r) = f \text{ and } (\sigma(r) \in S_1, \ \tau(r)^f \in S_1 \text{ or } \sigma(r) \in S_1 y, \ \tau(r)^f \in S_1 y) \}.$$

\begin{lemma}\label{Lemma5}
Let $h \in K^{\#}$ and $f = h^2$. Then
 the following statements hold.
\begin{enumerate}
\item If $y^h \in S_1$ and $y^{h^{-1}} \in S_1$, then
$$a^h, \quad y^h, \quad (y^{-1})^h \in {\mathcal M}_{gf}^1.$$
\item If $y^h \not\in S_1$ and $y^{h^{-1}} \in S_1$, then
$$a^{h} \in {\mathcal M}_{ygf}^1, \quad (y^{-1})^{h} \in {\mathcal M}_{gf}^2, \quad y^h \in {\mathcal M}_{ygf}^2.$$
\item If $y^h \in S_1$ and $y^{h^{-1}} \not\in S_1$, then
$$a^{h}, \quad y^h \in {\mathcal M}_{ygf}^1, \quad (y^{-1})^{h} \in {\mathcal M}_{gf}^1.$$
\item If $y^h \not\in S_1$ and $y^{h^{-1}} \not\in S_1$, then
$$a^{h} \in {\mathcal M}_{gf}^1, \quad y^h, \quad (y^{-1})^{h} \in {\mathcal M}_{gf}^2.$$
\end{enumerate}Moreover,
$${\mathcal M}_{gf}^1 \cup {\mathcal M}_{gf}^2 \neq \varnothing \iff {\mathcal M}_{ygf}^1 \cup {\mathcal M}_{ygf}^2 \neq \varnothing.$$
\end{lemma}

\begin{proof}
Fix an arbitrary element $h \in K$ and denote $f = h^2$.
Let $x_1 = a^{h}$, $x_2 = (y^{-1})^{h}$, and $x_3 = y^{h}$. Then, $\eta(x_i) = h^2 = f$ for all $i=1,2,3$, and,  by Lemma~\ref{Lemma4},
$$\sigma(x_1) = (y^{-1})^{h^{-1}} = (a y)^{h^{-1}} \in S_1 y^{h^{-1}}, \quad \tau(x_1)^f = y^{h^{-1}f} \in S_1 y^{h},$$
$$\sigma(x_2) = a^{h^{-1}} \in S_1, \quad \tau(x_2)^f = (y a)^{h^{-1}f} = (y a)^{h} \in S_1 y^{h},$$
$$\sigma(x_3) = y^{h^{-1}}, \quad \tau(x_3)^f = a^{h^{-1}f} = a^{h} \in S_1.$$

\medskip

I. First, we prove statements $(1)$--$(4)$.

\begin{enumerate}
\item

Suppose $y^{h^{-1}} \in S_1$.
Then we have $(y^{-1})^{h^{-1}} \in S_1$ and, by Lemma~\ref{Lemma4},
$$\sigma(x_1) \in S_1, \quad \tau(x_1)^f = y^{h^{-1}f} \in S_1 y^{h},$$
$$\sigma(x_2) \in S_1, \quad \tau(x_2)^f \in S_1 y^{h},$$
$$\sigma(x_3) \in S_1, \quad \tau(x_3)^f \in S_1.$$

\begin{enumerate}
 \item
Suppose $y^h \in S_1$. Then   $\tau(x_1)^f \in S_1$ and $\tau(x_2)^f \in S_1$. In this case,
$$x_1, x_2, x_3 \in {\mathcal M}_{gf}^1.$$

\item If $y^h \not\in S_1$, then $\tau(x_1)^f \in S_1 y$ and $\tau(x_2)^f \in S_1 y$, hence
$$x_1 \in {\mathcal M}_{ygf}^1, \quad x_2 \in {\mathcal M}_{gf}^2, \quad x_3 \in {\mathcal M}_{ygf}^2.$$
\end{enumerate}

\item Suppose $y^{h^{-1}} \in S_1 y$.
Then $(y^{-1})^{h^{-1}} \in S_1 y$, and,  by Lemma~\ref{Lemma4},
$$\sigma(x_1) \in S_1y, \quad \tau(x_1)^f = y^{h^{-1}f} \in S_1y^{h},$$
$$\sigma(x_2) \in S_1, \quad \tau(x_2)^f \in S_1y^{h},$$
$$\sigma(x_3) \in S_1y, \quad \tau(x_3)^f \in S_1.$$

\begin{enumerate}
 \item
If $y^h \in S_1$, then $\tau(x_1)^f \in S_1$ and $\tau(x_2)^f \in S_1$, hence
$$x_1, x_3 \in {\mathcal M}_{ygf}^1, \quad x_2 \in {\mathcal M}_{gf}^1.$$
\item If $y^h \not\in S_1$, then  $\tau(x_1)^f \in S_1 y$ and $\tau(x_2)^f \in S_1 y$, hence
$$y^{-1} x_2, y^{-1} x_3 \in S_1, \quad x_1 \in {\mathcal M}_{gf}^1, \quad x_2, x_3 \in {\mathcal M}_{gf}^2.$$
\end{enumerate}
\end{enumerate}
Thus the condition ${\mathcal M}_{ygf}^1 = {\mathcal M}_{ygf}^2 = \varnothing$ implies either $y^{h}, y^{h^{-1}} \in S_1$, or $y^{h}, y^{h^{-1}} \in S_1 y$.
\medskip

II. Next we prove that the case ${\mathcal M}_{ygf}^1= {\mathcal M}_{ygf}^2= \varnothing$ is impossible.

Recall that the group $S$ contains exactly $q+1$ elements $s$ such that $\eta(s)=f$. Moreover, all $q-2$ such elements $s$, distinct from $x_i$, have the form $s=(y(y^{-1})^k)^t$ for some $t,k \in K$, where $k\ne 1$, and belong to distinct $K$-orbits on $S$ (see \cite[Section 12]{Suz62}).
Put $\pi(z)= y(y^{-1})^z$ for $z \in K^{\#}$.
Since
\begin{equation}\label{lambda}z^{-1}a^zaz=a^x
\end{equation} for some $x \in K^{\#}$, then
$$\sigma(\pi(z))= y(y^{-1})^{x^{-1}z^2}, \quad \tau(\pi(z))= y^{x^{-1}}(y^{-1})^{z^{-1}}, \quad \text{and} \quad \eta(\pi(z))= x^2 z^{-2}.$$
Due to the regularity of the action (by conjugation) of $K$ on $Z(S)^{\#}$, the element $x$ is uniquely determined by $z$.
Then by Lemma~\ref{Lemma3} and the above formulas, for the element $s'=\pi(z)^{x^{-1}zh}$, we have $\eta(s')=f$,
$$\sigma(s')= \sigma(\pi(z))^{x z^{-1} h^{-1}}= y^{x z^{-1} h^{-1}} (y^{-1})^{z h^{-1}} \quad \text{and} \quad \tau(s')= \tau(\pi(z))^{x z^{-1} h^{-1}}= y^{z^{-1} h^{-1}} (y^{-1})^{z^{-2} x h^{-1}}.$$
This yields
$$\tau(s')^f = y^{z^{-1} h}(y^{-1})^{z^{-2} x h} = \left( y (y^{-1})^{x z^{-1}} \right)^{z^{-1} h}.$$

Now, let us apply the well-known $4$-dimensional matrix representation of the group $M=S:K$ over the field with $q=2^{2e+1}$ elements (see  \cite[Section 13]{Suz62}).
Let $w^2=2q$ and $\theta$ be the automorphism of the field $\F$, defined by the rule
$$
\theta: \alpha \mapsto \alpha^{w}.
$$
It is clear that $\theta^2$ maps the element $z$ to $z^2$.
Then, each element $s$ of the group $S$ corresponds to a matrix $(\alpha_{ij})$, denoted by $(\alpha,\beta)$, where $\alpha,\beta\in \F$, defined as follows:
$\alpha_{ij}=0$ if $i<j$, $\alpha_{ii}=1$, $\alpha_{21}=\alpha_{43}=\alpha$, $\alpha_{32}=\alpha^{\theta}$, $\alpha_{42}=\beta$,
$\alpha_{31}=\alpha^{1+\theta}+\beta$, $\alpha_{41}=\alpha^{2+\theta}+\alpha\beta+\beta^{\theta}$. In what follows, we identify such an element $s$ of $S$
with the corresponding pair $(\alpha,\beta)$.
Multiplication of elements $(\alpha,\beta)$ and $(\gamma,\delta)$ of the group $S$ is given by the rule
$$(\alpha,\beta) \cdot (\gamma,\delta) = (\alpha+\gamma, \alpha \gamma^{\theta} + \beta + \delta).$$
Also, each element from $Z(S)^{\#}$ corresponds to a matrix of the form $(0,\beta)$, where $\beta \in \F^*$.
We may assume further that the elements $y, y^{-1}$, and $a$ correspond to the matrices $(1,0)$, $(1,1)$, and $(0,1)$, respectively.

Next, to each element $\kappa$ in $\F^*$ corresponds a diagonal matrix $k(\kappa)$ with the elements $\zeta_1,\zeta_2,\zeta_3,\zeta_4$ in the main diagonal such that
$\zeta_1^{\theta}=\kappa^{1+\theta}$, $\zeta_2^{\theta}=\kappa$, $\zeta_3=\zeta_2^{-1}$ and $\zeta_4=\zeta_1^{-1}$. So conjugation of $(\alpha,\beta)$ by $k(\kappa)$
is given by the rule
$$k(\kappa)^{-1}(\alpha,\beta)k(\kappa) = (\kappa \alpha,\ \kappa^{1+\theta} \beta).$$

Since $Z(S)=S'\le S_1$ and $[S:S_1]=2$, there is an
$\mathbb F_2$-hyperplane $H\le\F$ such that
$$ S_1=\{(\alpha,\beta)\mid \alpha\in H,\ \beta\in\F\}.$$
Since $y=(1,0)\notin S_1$, we have $1\notin H$.
Choose a generator $\zeta$ of $\F^*$ and put $E=\zeta^{-1}H$.
Then $\dim_{\mathbb F_2}E=2e$, $H=\zeta E$, and
$$ S_1=\{(\zeta\varepsilon,\zeta^{\theta+1}\beta)\mid\varepsilon\in E,\ \beta\in\F\}.$$

Applying the above formulas, we obtain
 $$S_1 y = \{ (\zeta \alpha + 1,\ \zeta \alpha + \zeta^{\theta + 1} \beta) \mid \ \alpha \in \E,\ \beta \in \F \}.$$

For the element $h = k(\gamma) \in K^{\#}$, we have
\begin{equation*}
\begin{aligned}
y^{h^{-1}} &= (\gamma^{-1}, 0),\\
y^{h} &= (\gamma, 0),\\
 y \cdot (y^{-1})^{h}& = (1, 0)(1, 1)^{h} = (1, 0)(\gamma, \gamma^{1 + \theta}) = (1 + \gamma, \gamma^{\theta} + \gamma^{1 + \theta}).
\end{aligned}
\end{equation*}

For the element $z = k(\kappa) \in K^{\#}$, we have
\begin{equation}\label{eqpi}
\pi(z) = y \cdot (y^{-1})^{z} = (1 + \kappa, \kappa^{\theta} + \kappa^{1 + \theta}),
\end{equation}and
\begin{equation}\label{eqsigma}
\begin{aligned}
\sigma(\pi(z)) &= (1 + \lambda^{-1} \kappa^{2}, (\lambda^{-1} \kappa^{2})^{\theta} (1 + \lambda^{-1} \kappa^{2})),\\
\tau(\pi(z)) &= (\lambda^{-1} + \kappa^{-1}, \kappa^{-\theta} (\lambda^{-1} + \kappa^{-1})),
\end{aligned}
\end{equation}
where $\lambda$  is uniquely determined by the relation
\begin{equation}\label{eqlambda}
\lambda^{\theta + 1} = \kappa^{1 + \theta} + \kappa^{2 (1 + \theta)},
\end{equation} since
$x = k(\lambda)$ satisfies equation \eqref{lambda}.
Let us denote by $\varphi$ the map that assigns to the element $\pi(z)$ the element $\pi(z)^u$, where $u \in K$, such that $\eta(\pi(z)^u) = f$. Then, setting $\rho = \lambda \kappa^{-1}$, equation \eqref{eqlambda} is rewritten as
\begin{equation}\label{eqrho}\rho^{\theta+1} = \kappa^{\theta+1}\ + 1,
\end{equation} and, for the element
\begin{equation}\label{eqs'} s' := \pi(z)^{x^{-1} z h} = (\rho^{-1} \gamma (1 + \kappa), (\rho^{-1} \gamma)^{\theta+1} (\kappa^{\theta} + \kappa^{1 + \theta})),\end{equation}
we have $\eta(s') = f$, i.e., $\varphi(\pi(z)) = s'$.
Also,
\begin{equation}\label{eqsigma1}
\begin{split}
\sigma(s') &= \sigma(\pi(z))^{x z^{-1} h^{-1}} = (\lambda \kappa^{-1} \gamma^{-1} (1 + \lambda^{-1} \kappa^{2}),\ \lambda \kappa^{\theta - 1} \gamma^{-(\theta + 1)} (1 + \lambda^{-1} \kappa^{2})) \\
&= (\gamma^{-1} (\rho + \kappa),\ \kappa^{\theta} \gamma^{-(\theta + 1)} (\rho + \kappa))
\end{split}
\end{equation}
and
\begin{equation}\label{eqtau1}
\begin{split}
\tau(s')^{f} &= \tau(\pi(z))^{x z^{-1} h^{-1} \cdot h^{2}} = (\lambda \kappa^{-1} \gamma (\lambda^{-1} + \kappa^{-1}),\ (\lambda \kappa^{-1} \gamma)^{\theta+1} \kappa^{-\theta} (\lambda^{-1} + \kappa^{-1})) \\
&= (\kappa^{-1} \gamma (1 + \rho),\ (\rho \gamma)^{\theta+1} \kappa^{-\theta-1} (\rho^{-1} + 1)).
\end{split}
\end{equation}

Then
$s' \in \mathcal{M}_{gf}^1$ ($s' \in \mathcal{M}_{ygf}^2$) if and only if
\begin{equation}\label{condAmgf1}
\rho^{-1}(1+\kappa) \in \gamma^{-1}\zeta E \quad (\text{respectively, } \rho^{-1}(1+\kappa) \in \gamma^{-1}(\zeta E+1))
\end{equation}
and one of the following conditions holds:
\begin{equation}\label{eqmgf1}
\begin{cases}
\rho + \kappa \in \gamma \zeta E \\
\kappa^{-1}(1+\rho) \in \gamma^{-1} \zeta E
\end{cases},\\
\end{equation}
\begin{equation}\label{eqmgf1'}
\begin{cases}
\rho + \kappa \in \gamma (\zeta E + 1) \\
\kappa^{-1}(1+\rho) \in \gamma^{-1} (\zeta E + 1)
\end{cases}
\end{equation}

Similarly,
$s' \in \mathcal{M}_{ygf}^1$ ($s' \in \mathcal{M}_{gf}^2$) if and only if
\begin{equation}\label{condAmgf2}
\rho^{-1}(1+\kappa) \in \gamma^{-1} \zeta E \quad (\text{respectively, } \rho^{-1}(1+\kappa) \in \gamma^{-1} (\zeta E+1))
\end{equation}
and one of the following conditions holds:
\begin{equation}\label{eqmgf2}
\begin{cases}
\rho + \kappa \in \gamma (\zeta E + 1) \\
\kappa^{-1}(1+\rho) \in \gamma^{-1} \zeta E
\end{cases},\\
\end{equation}
\begin{equation}\label{eqmgf2'}
\begin{cases}
\rho + \kappa \in \gamma \zeta E \\
\kappa^{-1}(1+\rho) \in \gamma^{-1} (\zeta E + 1)
\end{cases}
\end{equation}

In the rest of the proof, we fix an arbitrary $\gamma\in F^*$ and suppose that ${\mathcal M}_{ygf}^i=\varnothing$ for all $i=1,2$, where $f=h^2$ and $h=k(\gamma)$. Then each element $s_1$ in $S_1^{\#}$ with $\eta(s_1)=f$ is contained in ${\mathcal M}_{gf}^1$, and each element $s_2$ in $S_1 y$ with $\eta(s_2)=f$ is contained in  ${\mathcal M}_{gf}^2$. Moreover, in this case, by the first part  of the lemma, we have either
$$\gamma, \gamma^{-1} \in \zeta E,$$
or
$$ \gamma, \gamma^{-1} \in \zeta E + 1.$$
In any case, $\gamma + \gamma^{-1} \in \zeta E$.

Next, consider the conditions:
\begin{equation}\label{rho^{-1}(1+kappa)}
s' \in S_1^{\#} \iff \rho^{-1}(1+\kappa) \in \gamma^{-1} \zeta E,
\end{equation}
\begin{equation}\label{rho+kappa}
\rho + \kappa \in \gamma \zeta E,
\end{equation}
\begin{equation}\label{kappa^{-1}(1+rho)}
\kappa^{-1}(1+\rho) \in \gamma^{-1} \zeta E.
\end{equation}

For $z_1 = z^{-1} x$, consider the element
$$\pi(z_1) = y (y^{-1})^{z^{-1} x} = (1 + \kappa^{-1} \lambda, \cdot).$$
We have
$$\kappa_1 = k(z_1) = \kappa^{-1} \lambda = \rho, \quad x_1 = x, \quad \rho_1 = \kappa,$$
and
\begin{equation*}\begin{aligned}
\pi(z_1)^{z_1 x_1^{-1} h} &= (\gamma \rho_1^{-1} (1 + \kappa_1), \cdot) = (\gamma \kappa^{-1} (1+\rho), \cdot),\\
\sigma(\pi(z_1)^{z_1 x_1^{-1} h}) &= (\rho' \gamma^{-1} (1 + \rho'^{-1} \kappa'), \cdot) = (\gamma^{-1} (\kappa + \rho), \cdot),\\
\tau(\pi(z_1)^{z_1 x_1^{-1} h})^f &= (\gamma \rho^{-1} (1 + \kappa), \cdot),\\
\eta(\pi(z_1)^{z_1 x_1^{-1} h})& = f.
\end{aligned}
\end{equation*}
It follows that $$\varphi(\pi(z_1)) = \pi(z_1)^{z_1 x_1^{-1} h} = \pi(z^{-1} x)^{z^{-1} h}.$$

Thus we shown the following claim.
\begin{stat}\label{claim} For each $\tilde{z}, \tilde{x}\in K$ and $\tilde{z}_1=\tilde{z}^{-1}\tilde{x}$ such that $\tilde{z}\ne 1$ and $a^{\tilde{x}}=\tilde{z}^{-1}{aa^{\tilde{z}}}\tilde{z}$, the following statements are true.
\begin{enumerate}
 \item[(i)]   $\varphi(\pi(\tilde{z}))$ ($\in S_1$) satisfies condition \eqref{eqmgf1} if and only if $\varphi(\pi(\tilde{z}_1))$   ($\in S_1$) satisfies    condition \eqref{eqmgf1};

 \item[(ii)]  $\varphi(\pi(\tilde{z}))$ ($\in S_1$) satisfies condition \eqref{eqmgf1'} if and only if $\varphi(\pi(\tilde{z}_1))$  ($\in S_1y$) satisfies condition \eqref{eqmgf2};

 \item[(iii)]  $\varphi(\pi(\tilde{z}))$ ($\in S_1 y$) satisfies condition \eqref{eqmgf2} if and only if $\varphi(\pi(\tilde{z}_1))$  ($\in S_1$) satisfies \eqref{eqmgf1'};

 \item[(iv)]  $\varphi(\pi(\tilde{z}))$ ($\in S_1 y$) satisfies condition \eqref{eqmgf2'} if and only if $\varphi(\pi(\tilde{z}_1))$  ($\in S_1y$) satisfies  condition \eqref{eqmgf2'}.
 \end{enumerate} (Note that here we use notation for $\kappa$ and $\lambda$ so that $k(\kappa)=\tilde{z}$, $k(\lambda)=\tilde{x}$ and $\rho:=\lambda\kappa^{-1}$).
 \end{stat}
\medskip
Now, for $z_2=z^{-1}$, consider the element $$\pi(z_2)=y(y^{-1})^{z^{-1}}=(1+\kappa^{-1},\cdot).$$
We have $a^{x_2}=za^{z^{-1}}az^{-1}=(z^{-1}a^zaz)^{z^{-3}}=c^{z^{-3}}$ when $x_2=xz^{-3}$, hence $$(\kappa_2,\rho_2)=(\kappa^{-1},\kappa^{-1}\rho),$$
and
\begin{equation*}\begin{aligned}
\pi(z_2)^{z_2x_2^{-1}h}&=(\gamma\rho^{-1}(1+\kappa),\cdot),\\
\sigma(\pi(z_2)^{z_2x_2^{-1}h})&=(\kappa^{-1}\gamma^{-1}(1+\rho),\cdot),\\
\tau(\pi(z_2)^{z_2x_2^{-1}h})^f&=(\gamma(\rho+\kappa),\cdot),\\
\eta(\pi(z_2)^{z_2x_2^{-1}h})&=f.\end{aligned}
\end{equation*}
Hence $\varphi(\pi(z_2))=\pi(z_2)^{z_2x_2^{-1}h}=\pi(z^{-1})^{z^2x^{-1}h}$.
On the other hand, let us consider the element $s'^{-1}$.
We have $\pi(z)^{-1}=y^zy^{-1}=\pi(z^{-1})^z$, and
$$s'^{-1}=(\pi(z)^{-1})^{x^{-1}zh}=\pi(z^{-1})^{x^{-1}z^2h}=(\rho^{-1}\gamma(\kappa+1),\cdot),$$
and
$$\eta(s'^{-1})=\eta(s')=f=\eta(\pi(z^{-1})^{x^{-1}z^2h})=fx^{-2}z^{4}\eta(\pi(z^{-1})).$$
Therefore, $(xz^{-2})^2= \eta(\pi(z^{-1})).$
Thus $\varphi(\pi(z^{-1}))=s'^{-1}=\varphi(\pi(z))^{-1}$.
Hence
$$\varphi(\pi(z))\in S_1 \iff \varphi(\pi(z^{-1}))\in S_1.$$

\medskip

Next, for $z_3=zx^{-1}=z_1^{-1}$, consider the element
$$\pi(z_3)=y(y^{-1})^{z_1^{-1}}=(1+\rho^{-1},\cdot).$$
We have $a^{x_3}= z_3^{-1}a^{z_3}az_3$ when $x_3=x^{-2}z^{3}$, so $$(\kappa_3,\rho_3)=(\kappa_1^{-1},\kappa_1^{-1}\rho_1)=(\rho^{-1},\rho^{-1}\kappa),$$ and
\begin{equation*}\begin{aligned}
\pi(z_3)^{z_3x_3^{-1}h}&=(\gamma\rho\kappa^{-1}(1+\rho^{-1}),\cdot)=(\gamma\kappa^{-1}(1+\rho),\cdot),\\
\sigma(\pi(z_3)^{z_3x_3^{-1}h})&=(\gamma^{-1}\rho^{-1}(\kappa+1),\cdot),\\
\tau(\pi(z_3)^{z_3x_3^{-1}h})^f&=(\gamma\rho(1+\rho^{-1}\kappa),\cdot)=(\gamma(\rho+\kappa),\cdot),\\
\eta(\pi(z_3)^{z_3x_3^{-1}h})&=f.\end{aligned}
\end{equation*}

Similarly, for $z_4 = z_2^{-1} x_2 = x z^{-2}$, consider the element
$$\pi(z_4) = y ( y^{-1} )^{x z^{-2}} = (1 + \rho \kappa^{-1}, \cdot).$$
We have $a^{x_4} = z_4^{-1} a^{z_4} a z_4$ with $x_4 = x z^{-3}$, hence
$$(\kappa_4, \rho_4) = (\rho_2, \kappa_2) = (\kappa^{-1} \rho, \kappa^{-1}),$$
and
\begin{equation*}\begin{aligned}
\pi(z_4)^{z_4 x_4^{-1} h}& = (\gamma \kappa (1 + \rho \kappa^{-1}), \cdot) = (\gamma (\kappa + \rho), \cdot),\\
\sigma(\pi(z_4)^{z_4 x_4^{-1} h}) &= (\gamma^{-1} \kappa^{-1} (1 + \rho), \cdot),\\
\tau(\pi(z_4)^{z_4 x_4^{-1} h})^f& = (\gamma \kappa \rho^{-1}(1 + \kappa^{-1}), \cdot) = (\gamma \rho^{-1} (\kappa + 1), \cdot)\\
\eta(\pi(z_3)^{z_3 x_3^{-1} h})& = f.
\end{aligned}
\end{equation*}

Finally, for $z_5 = z_3^{-1} x_3 = x^{-1} z^{2} = z_4^{-1}$, consider the element
$$\pi(z_5) = y ( y^{-1} )^{x^{-1} z^{2}} = (1 + \kappa \rho^{-1}, \cdot).$$
We have $a^{x_5} = z_5^{-1} a^{z_5} a z_5$ with $x_5 = x^{-2} z^{3}$, so
$$(\kappa_5, \rho_5) = (\rho_3, \kappa_3) = (\rho^{-1} \kappa, \rho^{-1}) = (\kappa_4^{-1}, \kappa_4^{-1} \rho_4),$$
and
\begin{equation*}\begin{aligned}
\pi(z_5)^{z_5 x_5^{-1} h} &= (\gamma \rho (1 + \kappa \rho^{-1}), \cdot) = (\gamma (\rho + \kappa), \cdot),\\
\sigma(\pi(z_5)^{z_5 x_5^{-1} h}) &= (\gamma^{-1} \rho^{-1} (\kappa + 1), \cdot),\\
\tau(\pi(z_5)^{z_5 x_5^{-1} h})^f &= (\gamma \rho \kappa^{-1}(1 + \rho^{-1}), \cdot) = (\gamma \kappa^{-1} (\rho + 1), \cdot),\\
\eta(\pi(z_5)^{z_5 x_5^{-1} h}) &= f.
\end{aligned}
\end{equation*}
\medskip

Denote by $\psi$ the mapping that assigns to each element $\varphi(\pi(\tilde{z}))$, where $\tilde{z} \in K^{\#}$, the element $\varphi(\pi(\tilde{z}^{-1} \tilde{x}))$, where $\tilde{x} \in K$, which is uniquely determined by $\tilde{z}$ via the relation $\tilde{z}^{-1} a^{\tilde{z}} a \tilde{z} = a^{\tilde{x}}$. Also, denote by $\varepsilon$  the mapping that assigns to each element $\varphi(\pi(\tilde{z}))$, where $\tilde{z} \in K^{\#}$, the element $\varphi(\pi(\tilde{z}^{-1}))$.
\medskip

Now let us consider the sextet $\{s_0,s_1,\ldots,s_5\}$ of elements in $S$, defined as follows:
$$s_0=\varphi(\pi(z)),\quad s_1=\psi(s_0)=\varphi(\pi(z^{-1}x)),\quad s_3=\varepsilon(s_1)=\varphi(\pi(zx^{-1})),$$
$$s_2=\varepsilon(s_0)=\varphi(\pi(z^{-1})),\quad s_4=\psi(s_2)=\varphi(\pi(z^{-2}x)),\quad s_5=\varepsilon(s_4)=\psi(s_3)=\varphi(\pi(x^{-1}z^2)).$$

$$\xymatrix{
s_0 \ar[r]^{\psi} \ar[d]^{\varepsilon}  & s_1 \ar[r]^{\varepsilon} & s_3 \ar[d]^{\psi} \\
s_2\ar[r]^{\psi} & s_4\ar[r]^{\varepsilon} & s_5}$$
\medskip

From   the shown above, only four cases are possible for $\{s_0,s_1,\ldots,s_5\}$
(the notation "$s_i \in (x)$" below means that, for the element $s_i$, the condition $(x)$   of the Claim \ref{claim} applied to $s_i=\varphi(\pi(\tilde{z}))$ is satisfied).
They are as follows.
\medskip

\underline{Case 1:}
$$\xymatrix{
s_0 \in (10) \ar[r]^{\psi} \ar[d]^{\varepsilon}  & s_1\in (10) \ar[r]^{\varepsilon} & s_3 \in (11) \ar[d]^{\psi} \\
s_2\in (11) \ar[r]^{\psi} & s_4 \in (13) \ar[r]^{\varepsilon} & s_5\in (13)
}$$
In this case, we have:
\begin{equation*}
\left\{
\begin{aligned}
\gamma\rho^{-1}(\kappa+1) &\in \zeta E \\
\gamma^{-1}(\rho+\kappa) &\in \zeta E \\
\gamma\kappa^{-1}(\rho+1) &\in \zeta E \\
\end{aligned}
\right.,
\quad
\left\{
\begin{aligned}
\gamma\kappa^{-1}(\rho+1) &\in \zeta E \\
\gamma^{-1}(\rho+\kappa) &\in \zeta E \\
\gamma\rho^{-1}(\kappa+1) &\in \zeta E \\
\end{aligned}
\right.,
\quad
\left\{
\begin{aligned}
\gamma\kappa^{-1}(\rho+1) &\in \zeta E \\
\gamma^{-1}\rho^{-1}(\kappa+1) &\in \zeta E +1 \\
\gamma(\rho+\kappa) &\in \zeta E +1 \\
\end{aligned}
\right.,
\end{equation*}
\begin{equation*}
\left\{
\begin{aligned}
\gamma\rho^{-1}(\kappa+1) &\in \zeta E \\
\gamma^{-1}\kappa^{-1}(\rho+1) &\in \zeta E +1 \\
\gamma(\rho+\kappa) &\in \zeta E +1 \\
\end{aligned}
\right.,
\quad
\left\{
\begin{aligned}
\gamma(\rho+\kappa) &\in \zeta E +1 \\
\gamma^{-1}\kappa^{-1}(\rho+1) &\in \zeta E +1 \\
\gamma\rho^{-1}(\kappa+1) &\in \zeta E \\
\end{aligned}
\right.,
\quad
\left\{
\begin{aligned}
\gamma(\rho+\kappa) &\in \zeta E +1 \\
\gamma^{-1}\rho^{-1}(\kappa+1) &\in \zeta E +1 \\
\gamma\kappa^{-1}(\rho+1) &\in \zeta E \\
\end{aligned}
\right.,
\end{equation*}
which implies
$\rho + \kappa \in \gamma \zeta E \cap \gamma^{-1} (\zeta E + 1)$.
\medskip

\underline{Case 2:} For all $s_i$, $s_i\in \eqref{eqmgf1}$. In this case, $\rho + \kappa \in \gamma \zeta E \cap \gamma^{-1} \zeta E$.

\medskip

\underline{Case 3:}
$$\xymatrix{
s_0 \in (14) \ar[r]^{\psi} \ar[d]^{\varepsilon}  & s_1\in (14) \ar[r]^{\varepsilon} & s_3 \in (13) \ar[d]^{\psi} \\
s_2\in (13) \ar[r]^{\psi} & s_4 \in (11) \ar[r]^{\varepsilon} & s_5\in (11)
}
$$

In this case, we have
 \begin{equation*}
         \left\{\begin{aligned}
              \gamma\rho^{-1}(\kappa+1)&\in \zeta E+1   \\
           \gamma^{-1}(\rho+\kappa) & \in \zeta E\\
           \gamma\kappa^{-1}(\rho+1)&\in \zeta E+1
                \end{aligned}\right.,
\quad        \left\{\begin{aligned}
              \gamma\kappa^{-1}(\rho+1)&\in \zeta E +1     \\
           \gamma^{-1}(\rho+\kappa) & \in \zeta E\\
           \gamma\rho^{-1}(\kappa+1)&\in \zeta E+1
                \end{aligned}\right.,
   \quad                \left\{\begin{aligned}
              \gamma\kappa^{-1}(\rho+1)&\in \zeta E +1     \\
           \gamma^{-1}\rho^{-1}(\kappa+1) & \in \zeta E+1\\
           \gamma(\rho+\kappa)&\in \zeta E
                \end{aligned}\right.,
    \end{equation*}
 \begin{equation*}
         \left\{\begin{aligned}
              \gamma\rho^{-1}(\kappa+1)&\in \zeta E  +1    \\
           \gamma^{-1}\kappa^{-1}(\rho+1) & \in \zeta E+1\\
           \gamma(\rho+\kappa)&\in \zeta E
                \end{aligned}\right.,
\quad        \left\{\begin{aligned}
              \gamma(\rho+\kappa)&\in \zeta E   \\
           \gamma^{-1}\kappa^{-1}(\rho+1)& \in \zeta E+1\\
           \gamma\rho^{-1}(\kappa+1)&\in \zeta E+1
                \end{aligned}\right.,
   \quad                \left\{\begin{aligned}
              \gamma(\rho+\kappa)&\in \zeta E     \\
           \gamma^{-1}\rho^{-1}(\kappa+1) & \in \zeta E+1\\
           \gamma\kappa^{-1}(\rho+1)&\in \zeta E+1
                \end{aligned}\right.,
    \end{equation*}
which implies $\rho+\kappa \in \gamma \zeta E \cap \gamma^{-1} \zeta E$.

\medskip

\underline{Case 4:} For all $s_i$, $s_i\in \eqref{eqmgf2'}$.
In this case, $\rho+\kappa \in \gamma^{-1}(\zeta E+1)\cap \gamma \zeta E$.

\medskip

Put $H=\zeta E$. Choose $c\in\F^*$ such that
\[
 H=\ker\chi,\qquad
 \chi(t)=\operatorname{Tr}_{\F/\mathbb F_2}(ct).
\]
Since $1\notin H$, we have $\chi(1)=1$. The the assumed emptiness of
$\mathcal M_{ygf}^1$ and $\mathcal M_{ygf}^2$ implies
\begin{equation}\label{eq:global-parity}
 \chi\bigl(\alpha(s)+\alpha(\sigma(s))
                    +\alpha(\tau(s)^f)\bigr)=0,
\end{equation}
where $\alpha(s)$ denotes the first coordinate of $s$ with $\eta(s)=f$.

Put $d=\theta+1$ and
$$ Q=\{(\kappa,\rho)\in\F^2:
                  \kappa^d+\rho^d=1\},\qquad
 Q^+=\{\kappa+\rho:(\kappa,\rho)\in Q\}.$$
The map $t\mapsto t^d$ is a permutation on $\F$, since
$(\theta+1)(\theta-1)=2q-1\equiv1\pmod{q-1}$. Thus $|Q|=q$.
Then, each $t\in Q^+$ has exactly two preimages in $Q$, so
$|Q^+|=q/2$. Also, $0\notin Q^+$, and the two preimages of $1$
are $(0,1)$ and $(1,0)$.

Define
$$ \ell_+(x)=\chi(\gamma x),\qquad
 \ell_-(x)=\chi(\gamma^{-1}x),\qquad
 \nu(x)=\ell_+(x)+\ell_-(x)
           =\chi((\gamma+\gamma^{-1})x).$$

For a representative $(\kappa,\rho)$ of a sextet, put
$$ A=\frac{1+\kappa}{\rho},\qquad
 B=\kappa+\rho,\qquad
 C=\frac{1+\rho}{\kappa}.$$
Here $\kappa\rho\ne0$. The simultaneous conditions in the four
sextet cases give
$$
\begin{array}{c|c|c|c}
 &(\ell_+(A),\ell_+(B),\ell_+(C))
 &(\ell_-(A),\ell_-(B),\ell_-(C))
 &(\nu(A),\nu(B),\nu(C))\\ \hline
 \text{Case 1}&(0,1,0)&(1,0,1)&(1,1,1)\\
 \text{Case 2}&(0,0,0)&(0,0,0)&(0,0,0)\\
 \text{Case 3}&(1,0,1)&(1,0,1)&(0,0,0)\\
 \text{Case 4}&(1,1,1)&(0,0,0)&(1,1,1)
\end{array}
$$
In particular, in every case,$$ \nu(A)=\nu(B)=\nu(C).$$
We next verify the identity relating these three values.
Applying $\theta$ to \eqref{eqrho}, we obtain
$$ 1=\kappa^{\theta+2}+\rho^{\theta+2}.$$
Multiplying the original equation by $B=\kappa+\rho$, we obtain
$$
\begin{aligned}
 B
 &= (\kappa+\rho)
       (\kappa^{\theta+1}+\rho^{\theta+1})\\
 &= \kappa^{\theta+2}+\rho^{\theta+2}
       +\kappa\rho(\kappa^\theta+\rho^\theta)\\
 &=1+\kappa\rho B^\theta.
\end{aligned}
$$
It follows that
$$
\begin{aligned}
 A+C
 &=\frac{\kappa+\kappa^2+\rho+\rho^2}{\kappa\rho}\\
 &=\frac{B+B^2}{\kappa\rho}
  =\frac{B(B+1)}{\kappa\rho}
  =B^{\theta+1}=B^d.
\end{aligned}
$$
By linearity of $\nu$,
$$ \nu(B^d)=\nu(A+C)=\nu(A)+\nu(C)=0.$$

Similarly,
$$ B+A=C^d,\qquad B+C=A^d.$$
Indeed, these are the same identity applied to the transformed
pairs $(\kappa^{-1},\rho/\kappa)$ and
$(\rho^{-1},\kappa/\rho)$, respectively.
Consequently,
$$ \nu(A^d)=\nu(B^d)=\nu(C^d)=0.$$

The six parameter pairs in this sextet are
$$ \begin{gathered}
 (\kappa,\rho),\quad(\rho,\kappa),\quad
 (\kappa^{-1},\rho/\kappa),\\
 (\rho^{-1},\kappa/\rho),\quad
 (\rho/\kappa,\kappa^{-1}),\quad
 (\kappa/\rho,\rho^{-1}),
 \end{gathered}
$$
and their sums are, respectively,
$$ B,\quad B,\quad C,\quad A,\quad C,\quad A.$$
Since the sextets exhaust
$Q-\{(0,1),(1,0)\}$, we conclude that, for all $t\in Q^+-\{1\}$,
$$ \nu(t^d)=0.$$

It remains to show  $\nu(1)=0$.
 For  $s=a^h$, the three first coordinates of $s$, $\sigma(s)$ and $\tau(s)^f$ are
$(0,\gamma^{-1},\gamma)$. Hence
$$ \nu(1)=\chi(\gamma+\gamma^{-1})=0.$$
We have therefore proved
$$ \{t^d:t\in Q^+\}\subseteq\ker\nu-\{0\}.$$

Finally, $\gamma+\gamma^{-1}\ne0$, so $\nu$ is a nonzero linear functional.
Thus $|\ker\nu|=q/2$.
On the other hand, the power map $t\mapsto t^d$ is a permutation,
and $|Q^+|=q/2$. The last inclusion would therefore imply
$$ \frac q2
 =\bigl|\{t^d:t\in Q^+\}\bigr|
 \le|\ker\nu-\{0\}|
 =\frac q2-1,$$
a contradiction.
Therefore, for every $f\in K^\#$,
$\mathcal M_{ygf}^1\cup\mathcal M_{ygf}^2\ne\varnothing$.
This completes the proof.
 \end{proof}

\begin{lemma}\label{LemmaD=3}
The graph $\Gamma$ is antipodal and its diameter is $3$.
\end{lemma}
\begin{proof}
By the choice of elements $y$ and $g$, we have $$S_1\sim S_1 y g a \sim S_1 g \sim S_1 y \sim S_1 g a \sim S_1 y g \sim S_1,$$ therefore $d(S_1, S_1 g) = 2 = d(S_1 y, S_1 y g)$
and $d(S_1, S_1 y) = 3 = d(S_1 g, S_1 y g)$.

Now we will show that $d(S_1, S_1 f) = d(S_1, S_1 y f) = 3$ for all $f \in K^{\#}$. Since $d(\Gamma') = 3$, then $d_{\Gamma}(S_1, \{S_1 f, S_1 y f\}) = 3$, i.e., in $\Gamma$ there exists a path of length $3$ from $S_1$
to $S_1 y^{\delta} f$, where $\delta \in \{0,1\}$.

Note that in $S_1\setminus Z(S)$, there exists an element $s_1$ such that ${s_1}^{f^{-1}} \notin S_1$.
By contradiction, suppose there exists an element $f \in K^{\#}$ such that for every element $s_1 \in S_1 - Z(S)$, the inclusion ${s_1}^{f^{-1}} \in S_1$ holds.
Then $f$ lies in $N_K(S_1)$ and we may assume that it is its generator. Since $N_K(S_1)$ acts (by conjugation) faithfully and semi-regularly on $S/Z(S)$, the set $S_1/Z(S)$ is a union of certain $N_K(S_1)$-orbits under this action. Therefore  $|f|$ divides $\gcd(q/2 - 1, q - 1) = 1$. Contradiction.

Now  we have
$$S_1 f s_1 = S_1 {s_1}^{f^{-1}} f = S_1 y f$$
and
$$S_1 y f s_1 = S_1 y {s_1}^{f^{-1}} f = S_1 f.$$
Therefore, any path of length $3$ from $S_1$ to $S_1 y^{\delta} f$, where $\delta \in \{0,1\}$, is mapped by the element $s_1$ to a path (of the same length) from $S_1$ to $S_1 y^{\delta + 1} f$.

From the above proof and by the transitivity of $G$ on the vertices of  $\Gamma$, it follows that each fibre consists of vertices mutually at distance $3$ in $\Gamma$.

It remains to prove that any two non-adjacent vertices from different fibres have at least one common neighbour.
By $2$-transitivity of $G$ on the set of fibres of  $\Gamma$, it suffices to consider its fibres containing the vertices $S_1$ and $S_1 g$.

Since $\Gamma'$ is a $(q^2 + 1, q - 1, q + 1)$-cover, any two of its vertices from distinct fibres have exactly $q+1$ common neighbours.
In particular, this holds for the vertices $S$ and $S g f$ for any $f \in K^{\#}$, and the involution $g f$ fixes exactly one vertex $S x$ in
the $\mu$-subgraph ${\Gamma'}_1(S) \cap {\Gamma'}_1(S g f)$.
Since $G$ contains the only class of involutions, the involution $g f$ fixes both vertices $S_1 x$ and $S_1 y x$ in $\Gamma$.
From this, it follows that $d_{\Gamma}(S_1, S_1 g) = 2 = d_{\Gamma}(S_1 y, S_1 y g)$.

It remains to prove that $d_{\Gamma}(S_1, S_1 y g f) = 2 = d_{\Gamma}(S_1 y, S_1 g f)$ for all $f \in K^{\#}$.

Recall $\sigma, \tau: S^{\#} \to S$ and $\eta: S^{\#} \to K$ are the three functions defined by the rule \eqref{tausigmaeta}.
By the definition of   $\Gamma$, we have
$${\Gamma}_1(S_1) = \{ S_1 y g s \mid s \in S_1 \} \cup \{ S_1 g y s \mid s \in S_1 \}$$
and
$${\Gamma}_1(S_1 y) = \{ S_1 g s \mid s \in S_1 \} \cup \{ S_1 y g y s \mid s \in S_1 \}.$$

Note that, for an arbitrary $f \in K^{\#}$,
$${\Gamma'}_1(S) \cap {\Gamma'}_1(S g f) = \{ S g s \mid s \in S^{\#}, \ \eta(s) = f \}.$$
Moreover, for each $f \in K^{\#}$, there are exactly $q + 1$ elements $s \in S^{\#}$ such that $\eta(s) = f$.

Now we fix an arbitrary element $f \in K$ and consider the preimages of the subgraph ${\Gamma'}_1(S) \cap {\Gamma'}_1(S g f)$ in $\Gamma$.

Observe that, for each $s \in S_1^{\#}$,
$$y g s (g f)^{-1} = y \sigma(s) \eta(s) g \tau(s) f = y \sigma(s) \eta(s) f^{-1} g \tau(s)^f \in S_1 y g S_1 \cup S_1 g y S_1   \iff  s \in \mathcal{M}_{gf}^1.$$

Similarly, for each $s \in S_1$,
$$g y s (g f)^{-1} = \sigma(y s) \eta(y s) g \tau(y s) f = \sigma(y s) \eta(y s) f^{-1} g \tau(y s)^f \in S_1 y g S_1 \cup S_1 g y S_1  \iff  y s \in \mathcal{M}_{gf}^2.$$

Therefore,
$${\Gamma}_1(S_1) \cap {\Gamma}_1(S_1 g f) = \{ S_1 y g s \mid s \in \mathcal{M}_{gf}^1 \} \cup \{ S_1 g y s \mid y s \in \mathcal{M}_{gf}^2 \}.$$
Similarly, for each $s \in S_1^{\#}$,
$$y g s (y g f)^{-1} = y g s g f y^{-1} = y \sigma(s) \eta(s) f^{-1} g \tau(s)^f y^{-1} \in S_1 y g S_1 \cup S_1 g y S_1   \iff s \in \mathcal{M}_{ygf}^1.$$

And, for each $s \in S_1$,
$$g y s (y g f)^{-1} = g y s g f y^{-1} = 
\sigma(ys) \eta(ys) f^{-1} g \tau(ys)^f y^{-1} \in S_1 y g S_1 \cup S_1 g y S_1   \iff   y s \in \mathcal{M}_{ygf}^2.$$

Hence
$${\Gamma}_1(S_1) \cap {\Gamma}_1(S_1 y g f) = \{ S_1 y g s \mid s \in \mathcal{M}_{ygf}^1 \} \cup \{ S_1 g y s \mid y s \in \mathcal{M}_{ygf}^2 \}.$$

Now in view of Lemma \ref{Lemma5}, for all $f \in K^{\#}$,
$${\mathcal M}_{gf}^1 \cup {\mathcal M}_{gf}^2 \neq \varnothing \text{ and } {\mathcal M}_{ygf}^1 \cup {\mathcal M}_{ygf}^2 \neq \varnothing.$$
and it follows that any two non-adjacent vertices from different fibres of $\Gamma$ have at least one common neighbour.
Hence $d(\Gamma) = 3$, and the relation "to coincide or be at distance $3$" is an equivalence relation on the set of vertices of  $\Gamma$.
Consequently, $\Gamma$ is an antipodal $2(q-1)$-cover of the complete graph on $q^2 + 1$ vertices.
The lemma is proved.
\end{proof}

\begin{lemma}\label{Lemma1b} The order of the kernel $R$ of the action of the group $\aut(\Gamma)$ on the set of antipodal fibres of  $\Gamma$ is equal to $2$, and $\aut(\Gamma)\le \aut(Sz(q)) \times R$.
Moreover, $\aut(\Gamma)$ is arc-transitive.
\end{lemma}
\begin{proof}
By the preceding lemma,  $\aut(\Gamma)$ preserves $\Sigma$, and the subgroup $R$ is well-defined. Note that $R$ acts semiregularly on vertices. Indeed, an element of $R$ fixing a vertex fixes its unique neighbour in each other fibre, since the edges between any two fibres form a perfect matching. Connectivity
then implies that it fixes every vertex. Hence $|R|$ divides $2(q-1)$.

Since $S_1\trianglelefteq S$, $y^2\in S_1$, and $D^y=D$, the mapping $$w:S_1x\longmapsto S_1yx$$
defines an involutory automorphism of $\Gamma$. It commutes with the right multiplication action of $G$ and interchanges the two vertices
of every block of $\mathcal B$. Since each such block lies in an antipodal fibre, $w\in R$.

Let $O=O_{2'}(R)$. Then $O\mathrel{\mathrm{char}}R\trianglelefteq \aut(\Gamma)$, so $G$ acts on $O$ by conjugation. Suppose this action is nontrivial.
Then simplicity of $G$ implies that it is faithful of degree
$$|O|\le q-1<q^2+1,$$
contrary to the minimal permutation degree of $Sz(q)$ (see \cite[p.~117]{Suz62}). Hence $[O,G]=1$.

As $N_G(S_1)=S$, it follows that
$$O\le C_{\sym(\Omega)}(G)\simeq N_G(S_1)/S_1\simeq Z_2,$$
so $O=1$.

Since $w\in R$ and $|R|\mid 2(q-1)$, a Sylow $2$-subgroup $P$ of $R$ has order $2$. In particular, $P\le Z(N_R(P))$, so Burnside's
normal $2$-complement theorem gives a normal Hall $2'$-subgroup of $R$. This subgroup is $O_{2'}(R)=1$. Consequently,
$$ R=\langle w\rangle\simeq Z_2. $$

Now $R\trianglelefteq \aut(\Gamma)$, and its orbits are precisely the blocks of $\mathcal B$. Therefore $\aut(\Gamma)$ preserves $\mathcal B$ and induces
a group of automorphisms of $\Gamma/\mathcal B=\Gamma'$. The kernel of this action is exactly $R$: it contains $R$, and any element fixing every
block of $\mathcal B$ also fixes every fibre in $\Sigma$, as $\mathcal B$ refines $\Sigma$. Thus
$$ \aut(\Gamma)/R\simeq \aut(\Gamma)^{\mathcal B}\le\aut(\Gamma')\simeq\aut(G), $$ where the latter isomorphism stated in \cite{Tsi22c}.
In particular, $\widetilde{G}:=\soc(\aut(\Gamma))=R\times G$ is arc-transitive.

\medskip

Let    $H$ be the stabiliser of the vertex $S_1$ in $\widetilde{G}$.
Then $H = S_1 \cup S_1 y w$, $|H| = q^2$, and $\Gamma \simeq \Cos(\widetilde{G}, H, HwgH)$.
For each element $f \in K^{\#}$, we have that $S_1 \cap S_1^f$ is a subgroup of index $2$ in $S_1$, containing $Z(S)$,
and $H \cap H^f = (S_1 \cap S_1^f) \cup ((S_1 y \cap (S_1 y)^f) w)$ is a subgroup of index $2$ in $H$, which fixes pointwise
two blocks $\{S_1, S_1 y\}$ and $\{S_1 f, S_1 y f\}$ and has exactly two orbits of length $q^2/2$ on the neighbourhood of any vertex $S_1 x$ from these two blocks, say $\Delta_1^x$ and $\Delta_2^x$, such that the set of fibres intersecting $\Delta_i^x$ coincides with the set of fibres intersecting $\Delta_i$.

\end{proof}
\begin{remark}

Note that, the additional representation of $\Gamma$ as the coset graph $\Cos(\widetilde{G}, H, HwgH)$ does not seem to provide a simpler way to deduce that the diameter  is $3$. This property of $\Gamma$ instead appears to be a consequence of the internal structure of the group $G=Sz(q)$, discussed in the proof of Lemma~\ref{Lemma5}.
\end{remark}

\section{Acknowledgements}
The work  is  supported by the Ministry of Science and Higher Education of the
Russian Federation, agreement No 075-02-2026-1315, for the development of the regional scientific
and educational mathematical center “Ural Mathematical Center”.

 \end{document}